\documentclass[a4,12pt,a4paper]{amsart}  
  
  \usepackage{a4wide,amsfonts,amsmath,amssymb,amsthm,xcolor}
\usepackage{enumerate}
\usepackage{hyperref}
\usepackage[all]{xy}
\usepackage{mathtools}\mathtoolsset{showonlyrefs}

\theoremstyle{plain}
\newtheorem{theorem}{Theorem}[section]
\newtheorem{corollary}[theorem]{Corollary}
\newtheorem{definition}[theorem]{Definition}
\newtheorem{lemma}[theorem]{Lemma}
\newtheorem{proposition}[theorem]{Proposition}

\theoremstyle{definition}

\theoremstyle{remark}
\newtheorem{examp}[theorem]{Example}
\newtheorem{remar}[theorem]{Remark}

\newenvironment{example}{\begin{examp}}{\hfill $\Diamond$\end{examp}}
\newenvironment{remark}{\begin{remar}}{\hfill $\Diamond$\end{remar}}

\DeclareMathOperator{\gr}{Gr}
\DeclareMathOperator{\hgr}{\mathfrak{Gr}}

\DeclareMathOperator{\rank}{rank}
\DeclareMathOperator{\rk}{rk}

\newcommand{\fE}{\mathfrak E}
\newcommand{\fQ}{\mathfrak Q}
\newcommand{\fF}{\mathfrak F}
\newcommand{\gS}{\mathfrak S}

\newcommand{\cO}{\mathcal O}

\newcommand{\C}{{\mathbb C}}
\newcommand{\R}{{\mathbb R}}
\newcommand{\Q}{{\mathbb Q}}
\newcommand{\Rplus}{{\mathbb R}_{> 0}}
\newcommand{\mumin}{\mu_{\min}}
\title[Positivity for Higgs bundles]{POSITIVITY FOR HIGGS VECTOR BUNDLES: \\[5pt]  CRITERIA AND APPLICATIONS}
\begin{document} \maketitle
\markright{\sc U. Bruzzo, A. Capasso, B. Gra\~na Otero}
\thispagestyle{empty}



\thispagestyle{empty} \vspace{-3mm}
\begin{center}{\sc Ugo Bruzzo,$^{abcd}$  Armando Capasso$^{e}$ and Beatriz Gra\~na Otero$^{f}$} \\[5pt]
\small
$^a$ SISSA (International School for Advanced Studies),  \\ Via Bonomea 265, 34136 Trieste, Italy \\
$^b$ Departamento de Matem\'atica, UFPB  \\ (Universidade Federal da Para\'iba), Jo\~ao Pessoa, PB, Brazil \\
$^c$ INFN (Istituto Nazionale di Fisica Nucleare), Sezione di Trieste \\
$^d$ IGAP (Institute for Geometry and Physics), Trieste  \\
$^e$ Dipartimento di Matematica e Fisica, Universit\`a degli Studi Roma Tre,\\ Largo San Leonardo Murialdo 1,   00146 Roma, Italy\\
$^f$ Departamento de Matem\'aticas  and IUFFYM   (Instituto de F\'\i sica \\ 
Fundamental y Matem\'aticas), Universidad de Salamanca, \\
Plaza de la Merced 1-4, 37008 Salamanca, Spain \\[5pt] 
Email: {\tt bruzzo@sissa.it, armando.capasso@uniroma3.it, beagra@usal.es}
\end{center}

%
%
\bigskip
\begin{abstract}
Working in the category of smooth projective varieties over an algebraically closed field of characteristic 0, we review notions of ampleness and numerical nefness for Higgs bundles which ``feel" the Higgs field and formulate criteria of the Barton-Kleiman type for these notions. We give an application to minimal surfaces of general type that saturate the Miyaoka-Yau inequality, showing that their cotangent bundle is ample. This will use results by Langer that imply that also for varieties over algebraically closed field of characteristic zero the so-called Simpson system is stable.
\end{abstract}

\vfill\begin{minipage}{\textwidth} \small
\parbox{\textwidth}{\hrulefill} \\
Date: Revised 8 August 2023 \\
Corresponding author: Ugo Bruzzo
ORCID: 0000-0001-8726-1729,    0009-0001-5463-7221, 0000-0003-1666-2685\\
U.B.'s research was partly supported by Bolsa de Produtividade 313333/2020-3 of Brazilian CNPq, by  PRIN 2017  ``Moduli theory and birational classification" and INdAM-GNSAGA; B.G.O.'s research was partly supported by Grant PID2021-128665NB-I00 funded by MCIN/AEI/ 10.13039/501100011033,  by ``ERDF A way of making Europe,''  and Universidad de Salamanca through Programa XIII.

\end{minipage}

\section{Introduction}
Positivity of line bundles plays an important role in algebraic geometry, in the form of the notions of ample and numerically effective line bundle. The concept of ampleness was extended to vector bundles by Hartshorne in \cite{Hartshorne}, and that of numerically effective vector bundle was introduced by Campana and Peternell in \cite{C:P}. Useful criteria for the characterization of ample and numerically effective vector bundles appeared over the years, see \cite{L:RK}, in particular the so-called   Barton-Kleiman criterion.

In \cite{B:GO:1, B:GO:3} two of the present authors defined notions of ampleness and numerical effectiveness for Higgs bundles, using a generalization of the Grassmann bundle of a vector bundle, called the \emph{Higgs Grassmannian scheme,} introduced in \cite{B:HR}. The basic idea is to formulate positivity notions that ``feel'' the Higgs field, and reduce to the classical ones when the Higgs field is zero. The main purpose of this paper is to formulate criteria for ampleness and numerical effectiveness of Higgs bundles of the Barton-Kleiman type (see also \cite{B:B:G}). This is done in Section \ref{3}, while in Section \ref{2} we recall the main definitions about Higgs bundles and their positivity. We also review the notion of {\em curve semistability}, which will play an important role. We shall do all this for Higgs bundles on smooth projective varieties defined over an algebraically closed field of characteristic 0.

 In Section \ref{4} we give an application of the criteria for the ampleness of Higgs bundles to surfaces of general type. Miyaoka and Yau \cite{Miyaoka-Chern,Yau1} in 1977 independently proved  that the Chern classes of a minimal surface $X$ of general type over $\C$ satisfy the inequality 
 \begin{equation} \label{ineq}3c_2(X)\ge c_1(X)^2.
 \end{equation}The surfaces that saturate this inequality, i.e., such that $3c_2(X)= c_1(X)^2$, were studied by Yau \cite{Yau1};  they are quotients of the unit ball in $\C^2$ by the action of a finite group. Simpson in 1988  \cite{S:CT:0}  proved the inequality \eqref{ineq} using Higgs bundles; he proved that the vector bundle $E=\Omega^1_X\oplus \cO_X$, called the {\em Simpson system}, has a natural Higgs field and the resulting Higgs bundle is stable. The Bogomolov inequality for this Higgs bundle then yields the inequality \eqref{ineq}. Moreover Miyaoka in 1987 \cite{Miyaoka-Kataka} proved that the cotangent bundle of these surfaces is ample.
 
Again Miyaoka in 1985 \cite{Miyaoka-Kodaira} proved that this inequality holds for surfaces of general type over an algebraically closed field of characteristic 0. The work of Langer \cite{L:A}, who proved the Bogomolov inequality and the stability of the Simpson system in positive characteristic, allows one to extend Simpson's proof to this setting. Now, the Simpson system is actually ample as a Higgs bundle and one can apply our criterion of the Barton-Kleiman type to show that the cotangent bundle of $X$ is stable also in the case of a field of characteristic 0.

Finally we briefly sketch a partial generalization to higher-dimensional projective varieties.

\bigskip
\section{H-ample and H-nef Higgs bundles}\label{2}
Let $X$ be a smooth projective variety over an algebraically closed field $\mathbb K$ of characteristic 0. We recall that a Higgs sheaf
is a pair $\fF=(F,\phi)$, where $F$ is a coherent $\cO_X$-module, and $\phi\colon F \to F\otimes\Omega^1_X$ is an 
$\cO_X$-linear morphism such that the composition 
$$\phi\wedge\phi \colon F \xrightarrow{\phi} F\otimes\Omega^1_X  \xrightarrow{\phi\times\mathrm{id}}   F\otimes \Omega^1_X\otimes\Omega^1_X \to F\otimes\Omega^2_X$$
is zero. A Higgs bundle is a Higgs sheaf with $F$ locally free. Semistability and stability are defined as for vector bundles but only with reference to $\phi$-invariant subsheaves.

The notion of curve semistable Higgs bundle was introduced in \cite{B:HR,B:GO:3}, generalizing a class of ordinary vector bundles that was studied in \cite{N,B:HR}.

\begin{definition}\label{curvesemistable}
A Higgs bundle $\fE$ is {\em curve semistable} if, for every morphism $f\colon C \to X$ where $C$ is a smooth irreducible projective curve, the pullback bundle $f^\ast \fE$ is semistable.
\end{definition}

We consider the characteristic class $$\Delta (E) = c_2(E)- \frac{r-1}{2r} c_1(E)^2 \in A^2(X)_\Q,$$called the {\em discriminant} of $E$ (here $r=\rank E$).

\begin{theorem}{\rm \cite{B:HR}}
Let $\fE=(E, \phi)$ be a Higgs bundle on $X$, which is semistable with respect to some polarization. If $\Delta(E) = 0$ then $\fE$ is curve semistable. \label{thm:css}
\end{theorem}

The opposite implication, which holds in the case of ordinary bundles, for Higgs bundles is a conjecture that is still largely open; for a discussion and more or less recent result see \cite{B:LG,L:LG,B:L:LG,B:C,L:A,B:GO:HR}. 
 
We recall the main definitions concerning H-ample and H-nef Higgs bundles. Let $\fE=(E,\phi)$ be a rank $r$ Higgs bundle
, and let $0<s<r$ be an integer number. Let ${p_s\colon\gr_s(E)\to X}$ be the Grassmannian bundle parametrizing rank $s$ locally free quotients of $E$ (see~\cite{F:W}). Consider the short exact sequence of vector bundles over $\gr_s(E)$ 
\begin{equation}\label{UnivBndlSeq}
\xymatrix{
0\ar[r] & S_{r-s,E}\ar[r]^(.55){\psi} & p_s^\ast E\ar[r]^(.45){\eta} & Q_{s,E}\ar[r] & 0
},
\end{equation}
where $S_{r-s,E}$ is the universal rank $r-s$ subbundle of $p_s^\ast E$ and $Q_{s,E}$ is the universal rank $s$ quotient. One defines the closed subschemes $\hgr_s(\fE)\subset\gr_s(E)$ (the \emph{Higgs Grassmannian schemes}) as the zero loci of the composite morphisms
\begin{displaymath}
\left(\eta\otimes\text{Id}_{\Omega^1_X}\right)\circ p_s^\ast(\phi)\circ\psi\colon S_{r-s,E}\to Q_{s,E}\otimes p_s^\ast\Omega_X^1.
\end{displaymath}
The restriction of the previous sequence to $\hgr_s(\fE)$ yields a universal short exact sequence
\begin{displaymath}
\xymatrix{
0\ar[r] & \mathfrak{S}_{r-s,\fE}\ar[r]^(.55){\psi} & \rho_s^\ast\fE\ar[r]^(.45){\eta} & \fQ_{s,\fE}\ar[r] & 0,
}
\end{displaymath}
where $\fQ_{s,\fE}=Q_{s,E|\hgr_s(\fE)}$ is equipped with the quotient Higgs field induced by $\rho_s^\ast\phi$, where ${\rho_s=p_{s|\hgr_s(\fE)}}$.
The scheme $\hgr_s(\fE)$ enjoys the usual universal property: for a morphism of varieties $f\colon Y\to X$, the morphism $g\colon Y \to \gr_s(E)$ given by a rank $s$ quotient $Q$ of $f^\ast E$ factors through $\hgr_s(\fE)$ if and only if $\phi$ induces a Higgs field on $Q$.
\begin{definition}\cite[Definition 2.3]{B:GO:3}\label{def1.1}
A Higgs bundle $\fE=(E,\phi)$ of rank one is Higgs ample (H-ample for short) if $E$ is ample in the usual sense. If $\rank(\fE)\geq2$, one inductively defines H-ampleness by requiring that
\begin{enumerate}
\item the Higgs bundles ${\mathfrak Q}_{s,\fE}$ are H-ample for all $s$, and
\item the determinant line bundle $\det(E)$ is ample.
\end{enumerate}
\end{definition}
The condition on the determinant cannot be omitted as Example 2.5 in \cite{B:GO:1} shows.

\begin{remark}The notion of H-ampleness coincides with the usual ampleness when the Higgs field is zero.\end{remark}
\begin{remark}\label{rem1.1}
The recursive condition in this definition can be recast as follows. Let $\,1\leq s_1<s_2<\dotsc<s_k<r$ and let $\fQ_{{(s_1, \cdots, s_k)},\fE}$ be the rank $s_1$ universal Higgs quotient bundle over $\hgr_{s_1}\left(\fQ_{(s_2,\dotsc,s_k),\fE}\right)$, obtained by taking the successive universal Higgs quotient bundles of $\fE$ of rank $s_k$, then $s_{k-1}$, all the way to rank $s_1$. The H-ampleness condition for $\fE$ is equivalent to requiring that all line bundles $\det(\fE)$ and $\det\left(\fQ_{{(s_1, \cdots, s_k)},\fE}\right)$ are ample.
\end{remark}
\noindent We prove now some properties of H-ample Higgs bundles which will be useful in the sequel.
\begin{proposition}\label{prop1.1} Let $\fE$ be a Higgs bundle on a smooth projective variety $X$.

(1) Let $f\colon Y\to X$ be a finite morphism of smooth projective varieties. If $\fE$ is H-ample then $f^\ast\fE$ is H-ample. Moreover, if $f$ is also surjective and $f^\ast\fE$ is H-ample then $\fE$ is H-ample.

(2) If $\fE$ is H-ample then every quotient Higgs bundle of $\fE$ is H-ample.
\end{proposition}
\begin{proof}
(1) In the rank one case, this is standard. 
In the higher rank case, one first notes that $f^\ast\det(E)=\det(f^\ast E)$, so that the condition on the determinant is fulfilled. By the functoriality of the Higgs Grassmannians, for all $s\in\{1,\dotsc,r-1\}$ the morphism $f$ induces finite morphisms $f_s\colon\hgr_s\left(f^\ast\fE\right)\to\hgr_s(\fE)$  such that $\fQ_{s,f^\ast\fE}=f_s^\ast\fQ_{s,\fE}$. By induction on the ranks of the universal quotients one concludes.

(2) If $\fQ$ is a Higgs quotient of $\fE$, for every $s$ there is a closed embedding $\hgr_s(\fQ) \to \hgr_s(\fE)$ which pullbacks the universal Higgs quotients of $\fE$ to the universal quotients of $\fQ$. Then one concludes by the previous point.
\end{proof}

There is an analogous notion of Higgs numerically effective (H-nef) Higgs bundle \cite{B:GO:1, B:GO:3}. The analogue of Proposition \ref{prop1.1} for H-nef Higgs bundles is as follows.

\begin{proposition}\label{propnef}Let $\fE$ be a Higgs bundle on a smooth projective variety $X$.

(1) Let $f\colon Y\to X$ be a morphism of smooth projective varieties. If $\fE$ is H-nef then $f^\ast\fE$ is H-nef. Moreover, if $f$ is also surjective and $f^\ast\fE$ is H-nef then $\fE$ is H-nef.

(2) If $\fE$ is H-nef then every quotient Higgs bundle of $\fE$ is H-nef.
\end{proposition}

\bigskip
\section{Criteria for H-ampleness and H-nefness}\label{3}
$X$ will be a smooth projective variety of dimension $n$ over an algebraically closed field $\Bbb K$ of characteristic 0.  We recall that every Higgs sheaf $\fE$ has a Harder-Narasimhan filtration
\begin{equation}\label{HN} 0 = \fE_{0}\subset \fE_1 \subset \dots \subset \fE_{m-1} \subset \fE_m = \fE \end{equation}
and the slopes of the successive quotients $\fQ_i=\fE_i/\fE_{i-1}$ satisfy the condition
$$ \mu_{\mbox{\tiny max}}(\fE) := \mu(\fQ_1) > \dots > \mu(\fQ_m) := \mumin (\fE).$$
For every Higgs quotient $\fQ$ of $\fE$ one has $\mu(\fQ) \ge  \mumin (\fE)$.

Whenever we consider a morphism $f \colon C \to X$, we understand that $C$ is a smooth projective connected curve.

\begin{theorem} \label{critmumin}Let $\fE=(E,\phi)$ be a rank $r$ Higgs bundle on $X$. Fix an ample class $h\in N^1(X)$. 
Then $\fE$ is H-ample if and only if
\begin{enumerate}  \item the line bundle $\det (E)$ is ample;
\item there exists $\delta\in \Rplus$ such that for every finite morphism $f\colon C\to X$,
the inequality
\begin{equation}\label{mumin}
\mumin (f^\ast \fE) \ge \delta \int_Cf^\ast h
\end{equation}
holds.
\end{enumerate}
\end{theorem}

\begin{proof}
Let us assume that $\det (E)$ is ample and condition \eqref{mumin} holds. The H-ampleness of $\fE$ is equivalent to the ampleness of a collection of line bundles $L_S$, each on an iterated Higgs Grassmannian, which we denote generically by $S$, with projection $\pi_S \colon S \to X$ (these line bundles are obtained by successively taking the universal Higgs quotient until one reaches the rank one quotient bundles, see Remark \ref{rem1.1}). Let $q_S \colon \pi_S^\ast \fE \to L_S$ be the quotient morphism, let $g \colon C \to S$ be a finite morphism, and let $f = \pi_S \circ g$. We have a Higgs quotient $f^\ast \fE \to \fQ$, where $\fQ=g^\ast L_S$. By the properties of the Harder-Narasimhan filtration we have $$\deg g^\ast L_S = \deg \fQ \ge \mumin (f^\ast \fE) \ge \delta \int_Cf^\ast h ,$$so that $L_S$ is ample for every choice of $S$. As a consequence, $\fE$ is H-ample.

To prove the opposite implication, note that since $\fE$ is H-ample, the determinant $\det (E)$ is ample, and the class $c_1(E)$ is ample as well. Let $f\colon C \to X$ be a morphism. We have two cases, according to whether the Higgs bundle $f^\ast \fE$ is semistable or not. Let us start with the first case. By the Kleiman criterion (Corollary 1.4.11 of \cite{L:RK}), there exists $\epsilon \in \Rplus$ such that for every irreducible curve $\bar C$ in $X$ one has $$\frac{c_1(E)\cdot\bar C}{h \cdot \bar C}\ge \epsilon.$$Then $$\mumin (f^\ast \fE) = \mu (f^\ast \fE) = \frac1r\int_C f^\ast c_1(E)= \frac1r c_1(E) \cdot \bar C \ge \frac{\epsilon}{r} h \cdot \bar C= \delta \int_C f^\ast h,$$where $\bar C$ is the image of $C$ in $X$, and $\delta= \epsilon/r$.

In the second case, let \begin{displaymath}
0\subset\fE_1\subset\dotsc\subset\fE_{m-1}\subset\fE_m=f^\ast \fE
\end{displaymath}be the Harder-Narasimhan filtration of $f^\ast \fE$. (Note the change of notation with respect to eq.~\eqref{HN}). By the universality and functoriality  of the Higgs Grassmannian there is a lift $f_s \colon C \to \hgr_s(\fE)$ of $f$ such that $\fE_m/\fE_{m-1}= f_s^\ast \fQ_{s,\fE}$
. Therefore, $$\mumin (f^\ast \fE) = \frac 1s \int_C f_s^\ast (c_1(\fQ_{s, \fE})).$$Since $\fQ_{s, \fE}$ is H-ample, there exists $\eta \in \Rplus$ such that the class $$c_1(\fQ_{s, \fE})-\eta \rho_s^\ast h$$is ample for all possible values of $s$, so that $$ \frac 1s \int_C f_s^\ast (c_1(\fQ_{s, \fE})) \ge \delta \int_C f^\ast h,$$where $\delta=\eta/(r-1)$, thus proving the claim.
\end{proof}

\begin{corollary}[Barton-Kleiman-type criterion for H-ampleness]\label{critquots}
Let $\fE=(E,\phi)$ be a rank $r$ Higgs bundle on $X$. Fix an ample class $h\in N^1(X)$. 
Then $\fE$ is H-ample if and only if
\begin{enumerate}  \item the line bundle $\det (E)$ is ample;
\item there exists $\delta\in \Rplus$ such that for every finite morphism $f\colon C\to X$, and for every Higgs quotient $\fQ$ of $f^\ast \fE$, the inequality
\begin{equation}\label{quots}
\deg (\fQ) \ge \delta \int_Cf^\ast h
\end{equation}
holds.
\end{enumerate}
\end{corollary}
\begin{proof}
If $\fE$ is H-ample, by Theorem \ref{critmumin} there exists $\delta \in \Rplus$ such that $$\deg (\fQ) = s \mu(\fQ) \ge s \mumin(f^\ast \fE) \ge s \delta \int_C f^\ast h \ge \delta \int_Cf^\ast h$$(here $s= \rk(\fQ)$.)

Conversely, let us assume that $\fE$ satisfies the conditions (1) and (2); let us call $\epsilon$ the constant. In particular, we can take for $\fQ$ the quotient $\fE_m/\fE_{m-1}$ of the Harder-Narasimhan filtration of $f^\ast E$, so that $$\mumin (f^\ast \fE) = \frac{\deg(\fE_m/\fE_{m-1})}{s}\ge \frac\epsilon s\int_C f^\ast h \ge \delta \int_Cf^\ast h$$where $\delta=\epsilon/r$. Therefore, $\fE$ is H-ample.
\end{proof}

\begin{example}
Let $X$ be a smooth projective irreducible curve of genus $g$, and $L_1$, $L_2$ line bundles on $X$, of degrees $d_1$ and $d_2$, respectively. Let $E = L_1 \oplus L_2$, and assuming that $d_1 \le 2 g -2+d_2$, equip it with a nonzero, nilpotent Higgs field given by $$\phi \colon L_1 \to L_2 \otimes \Omega_X, \qquad \phi(L_2)=0$$ where $\phi\colon L_1 \to L_2 \otimes \Omega_X$ is any nonzero morphism.  The Higgs line bundle $\mathcal L_1=(L_1,0)$ is a Higgs quotient of
 $\fE=(E,\phi)$. Let $d_1 + d_2 > 0$ so that $\det (E)$ is ample. Then $\fE$ is H-ample if and only if $\deg L_1=d_1 > 0$. 
 
 If $\deg L_1> 0$, fix an ample divisor $h$ on $X$; then 
\begin{equation}\deg L_1 \ge \delta \deg h \label{onX} \end{equation} choosing $0< \delta \le \frac{d_1}{\deg h}$. 
Now if $f\colon C \to X$ is a finite morphism,  
according to the analysis in \cite{B:HR}, Section 3.3,
$f^\ast \fE$ has at most two locally free rank 1 Higgs quotients, i.e.,
the pullback $Q_1=f^\ast L_1$ with zero Higgs field, and possibly another quotient 
with underlying bundle $Q_2$ with  $\deg Q_2\ge \deg Q_1$.
Now 
\eqref{onX} implies
$$ \deg Q_2 \ge \deg Q_1 \ge \delta \int_C f^\ast h$$
so that $\fE$ is H-ample.

On the other hand, if $\fE$ is H-ample, then $L_1$ is ample, i.e., $d_1>0$, by Proposition \ref{prop1.1}  (2).
 
If $g\ge 3$ one can choose $d_2<0$ so that  $E$ is not ample as an ordinary bundle (for instance, with $g=3$ take $d_1=2$ and $d_2=-1)$. 
\end{example}


The versions of the two previous criteria for H-nef Higgs bundles are as follows. The first one was proved in \cite{B:B:G}, Lemma 3.3.

\begin{theorem} \label{thm:Hnefcrit}
Let $\fE$ be a rank $r$ Higgs bundle on $X$. Then $\fE$ is H-nef if and only if $\mumin(f^\ast \fE) \ge 0$ for all morphisms  $f \colon C \to X$.
\end{theorem}

\begin{remark} This theorem can be equivalently stated by considering all morphisms  $f \colon C \to X$, or only finite ones.
\end{remark}

\begin{proof}
A proof is given in \cite{B:B:G}. Here we only note that when $\mumin(f^\ast \fE) \ge 0$ the line bundle $\det (E)$ is nef since $$\deg f^\ast \det(E) = r \mu(f^\ast E) \ge r \mumin(f^\ast \fE) \ge 0$$ for every morphism $f \colon C \to X$.
\end{proof}

\begin{corollary}
Let $\fE$ be a rank $r$ Higgs bundle on $X$. Then $\fE$ is H-nef if and only if 
\begin{enumerate}  \item the line bundle $\det (E)$ is nef;
\item for every morphism $f\colon C\to X$, and every Higgs quotient $\fQ$ of $f^\ast \fE$, the inequality $
\deg (\fQ) \ge 0$ holds.
\end{enumerate}
\end{corollary}

We show that that the notion of H-ampleness is well-behaved with respect to tensor products and extensions.

\begin{theorem}\label{ampletensors}
If $\fE$ and $\fF$ are H-ample Higgs bundles, their tensor product $\fE \otimes \fF$ is H-ample as well.
\end{theorem}
\begin{proof}
Fix an ample class $h \in N^1(X) $. By Theorem \ref{critmumin} there exists $\delta_1$, $\delta_2$ such that for every morphism $f \colon C \to X$, the inequalities $$\mumin (f^\ast \fE) \ge \delta_1 \int_Cf^\ast h, \qquad \mumin (f^\ast \fF) \ge \delta_2 \int_Cf^\ast h$$ hold. Letting $\delta=\delta_1 + \delta_2$ we obtain $$\mumin (f^\ast (\fE \otimes \fF))= \mumin( f^\ast \fE) + \mumin( f^\ast \fF) \ge \delta \int_Cf^\ast h.$$
\end{proof}

\begin{theorem}\label{ampleext}
Let $0\to\fE_1\to\fE\to\fE_2\to0$ be a short exact sequence of Higgs bundles over $X$, where $\fE_1=(E_1,\phi_1)$ and $\fE_2=(E_2,\phi_2)$ are H-ample. Then $\fE$ is H-ample.
\end{theorem}
\begin{proof} First note that the line bundle $\det(E)\cong\det(E_1)\otimes\det(E_2)$ is ample. Let $f \colon C \to X$ be a finite morphism and let $\fQ$ be a quotient Higgs bundle of $f^\ast \fE$. We can form the following diagram with exact rows and columns:
\begin{displaymath}
\xymatrix{
0\ar[r] & f^{*}\fE_1\ar[r]\ar[d]& f^{*}\fE\ar[r]\ar[d] & f^{*}\fE_2\ar[r]\ar[d] & 0\\
0 \ar[r] & \fQ_1\ar[r]\ar[d] & \fQ\ar[d] \ar[r]& \fQ_2 \ar[d] \ar[r] & 0\\
 & 0 & 0 & 0
}
\end{displaymath}
Let $\fQ_2'$ be $\fQ_2$ modulo its torsion, and let $h \in N^1(X)$ be an ample class. By Corollary \ref{critquots} there exist $\delta_1,\delta_2\in\mathbb{R}_{>0}$ such that
$$\deg ( \fQ_1) \ge \delta_1 \int_Cf^\ast h, \qquad \deg ( \fQ_2) \ge \deg( \fQ'_2)\ge \delta_2 \int_Cf^\ast h.$$By letting $\delta= \delta_1+ \delta_2$ we have $\deg (\fQ) \ge \delta \int_C f^\ast h$. Again by Corollary \ref{critquots}, $\fE$ is H-ample.
\end{proof}

The analogue of Theorem \ref{ampletensors} for H-nef Higgs bundles was proved in \cite{B:B:G}. On the other hand,  the version of Theorem \ref{ampleext} for H-nef Higgs bundles is proved along the same lines.


\bigskip
\section{An application to surfaces of general type}\label{4}
In this section we give an application of the H-ampleness  and H-nefness criteria as expressed by Theorems \ref{critmumin}  and  \ref{thm:Hnefcrit} to projective surfaces of general type over an algebraically closed field $\Bbb K$ of characteristic 0. We recall that the Chern classes of a minimal surface of general type $X$ satisfy the Miyaoka-Yau inequality
\begin{equation}\label{MY}
3c_2(X) \ge c_1(X)^2.
\end{equation}
This was proved in \cite{Yau1, Miyaoka-Chern} over the complex numbers and in \cite{Miyaoka-Kodaira} over $\Bbb K$. Moreover, Miyaoka in \cite{Miyaoka-Kataka}, working over $\C$, proved that the cotangent bundle of a minimal surface of general type that saturates the inequality is ample. We extend this result for a field $\Bbb K$ as above relying on Theorems \ref{critmumin} and \ref{thm:Hnefcrit}. This will be based on the properties of the so called Simpson system, i.e., the Higgs bundle $\gS=(S, \phi)$, where $S=\Omega_X^1 \oplus \cO_X$, and $\phi(\omega, f)=(0, \omega)$. Simpson proved in  \cite{S:CT:0}, over $\C$,  that $\gS$ is stable, so that the Bogomolov inequality for Higgs bundles implies the inequality \eqref{MY}. Then in \cite{L:A}  Langer extended this result  when $\Bbb K$ is a field as above.

Actually, when the Miyaoka-Yau equality is saturated, the Simpson system is curve semistable.

\begin{proposition} If $X$ is a projective minimal surface of general type over $\mathbb K$ such that $3c_2(X)=c_1(X)^2$, then the
 Higgs bundle $\gS$ is curve semistable. \label{prop:Simpcss}
\end{proposition}
\begin{proof} This follows from \ref{thm:css} since
$$\Delta(S) = c_2(X) - \tfrac13 c_1(X)^2 = 0. $$
\end{proof}

\begin{theorem} Under the same hypotheses of Proposition \ref{prop:Simpcss}, 
 the Higgs bundle $\gS_\beta= \gS(-\beta K_X)$ is H-nef for every rational\footnote{As it is customary, we formally consider twistings by rational divisors, which make sense after pulling back to a (possibly ramified) finite covering of $X$; on the other hand, the properties of being semistable, H-ample, H-nef are invariant under such coverings.} $\beta\le \frac13$.
 \end{theorem}
 
\begin{proof} $\gS_\beta$ is curve semistable, so that for every morphism $f\colon C \to X$, the pullback Higgs bundle
$f^\ast \gS_\beta$ is semistable, hence $$\mu_{\text{\tiny min}}(f^\ast \gS_\beta) = \mu(f^\ast \gS_\beta)$$  and 
$$\mu_{\text{\tiny min}}(f^\ast \gS_\beta) = \tfrac13 \int_C f^\ast c_1(S_\beta ) = \tfrac13 C' \cdot c_1(S(-\beta K_X)) = 
(\tfrac13-\beta) C'\cdot K_X \ge 0 $$  where $C'=f(C)$ 
(the last inequality holding as $K_X$ is nef) so that the claim follows from the H-nefness criterion (Theorem \ref{thm:Hnefcrit}).
\end{proof}
\begin{corollary} \label{Canample}The bundle $\Omega_\beta=\Omega^1_X(-\beta K_X)$ is nef for every rational $\beta\le \frac13$. As a consequence,
$K_X$ is ample. 
 \end{corollary}
\begin{proof} $\Omega_\beta$  with the zero Higgs field  is a Higgs quotient of $\gS_\beta$, hence 
  it is H-nef  and then nef in the usual sense.
Let $C\subset X$ be an irreducible curve, $\tilde C$ its normalization, and $\iota\colon\tilde C\to X$. We have a morphism
 $$ \iota^\ast \Omega_\frac13 \to \cO_{\tilde C}(K_{\tilde C} - \tfrac13\iota^\ast K_X)  $$
 so that $ \cO_{\tilde C}(K_{\tilde C} - \tfrac13\iota^\ast K_X) $ is nef, i.e., its degree is nonnegative, and
  $$ 2g(C) - 2  \ge 2g(\tilde C) - 2\ge \tfrac13K_X\cdot C .$$
As a consequence   $X$ has no rational curves.
Since the canonical bundle of a smooth minimal surface of general type $Z$ is ample if and only if  $Z$ has no rational curves $C$ such that $K_Z\cdot C =0 $, it follows that  $K_X$ is ample.
 \end{proof}   
Now, 
$$ \Omega_X^1 \simeq \Omega_X^1(-\tfrac13K_X) \otimes \cO_X(\tfrac13K_X)$$
so that $ \Omega_X^1 $ is the tensor product of a nef bundle by an ample line bundle, and therefore is ample.
      It may instructive to deduce this from the H-ampleness criterion given by Theorem \ref{critmumin}.
    \begin{lemma} Again under the  hypotheses of Proposition \ref{prop:Simpcss}, the Higgs bundle $\gS$ (the Simpson system)  is H-ample.
    \end{lemma}
\begin{proof} Fix an ample class $h$. A Higgs bundle $\fE$  is H-ample if and only if its determinant is ample and there exists a $\delta\in\R_{>0}$ such that
    $$ \mu_{\text{\tiny min}} (f^\ast \fE) \ge \delta \int_C f^\ast h $$
    for all $f\colon C \to X$ as before.
    
    Note that $\det(S)= K_X$ is ample by Corollary \ref{Canample}. Take $h=K_X$ and $\delta = \frac13$. Since $\gS$ is curve semistable
    $$  \mu_{\text{\tiny min}} (f^\ast \gS)  = \mu(f^\ast \gS)  = \tfrac13 \int_C f^\ast c_1(S) = \delta  \int_C f^\ast K_X.$$     
 \end{proof}
Note that $E$, as an ordinary bundle, is neither semistable nor ample.

Again, $\Omega_X^1$ with the zero Higgs field is a Higgs quotient of $\gS$, so that it is ample.

In the higher dimensional case one does not know if the Simpson system $\gS$  is semistable  with respect to some polarization $H$. If we assume that $X$ is a smooth projective variety of dimension $n$,   
that the associated Simpson is semistable, and also suppose that
\begin{equation}\label{higherMY}
\left[c_2(X)-\frac{n}{2(n+1)}c_1(X)^2\right]\cdot H^{n-2} =0,
\end{equation} and that the canonical class $K_X$ is ample,  then the same argument works to show that $\Omega_X^1$ is ample.

More generally, we have:

\begin{proposition}\label{higherdim}
Let $X$ be a smooth projective variety of dimension $n$ whose canonical bundle $K_X$ is nef. Assume that its Simpson system is semistable with respect to some polarization $H$, and the equality \eqref{higherMY} holds. Then
\begin{enumerate}
\item The Simpson system of $X$ is curve semistable and H-nef.
\item The twisted cotangent bundle $\Omega_\beta=\Omega_X^1(-\beta K_X)$ is nef for every rational $\beta \le \frac{1}{n+1}$. In particular the cotangent bundle $\Omega_X^1$ is nef.
\item $X$ has no irreducible curves of genus 0.
\item If $K_X$ is ample, then $\Omega^1_X$ is ample, and $X$ has no curves of genus 0 and 1.
\end{enumerate} 
\end{proposition}
Item (3) should be compared with the result that the canonical bundle of a smooth projective variety that contains no  rational curves is  nef, 
see e.g.~Theorem 1.16 in \cite{Debarre-Gael}.

\bigskip
\noindent{\bf Statement about competing or financial interests.} The authors have no competing or financial  interests to declare that are relevant to the content of this article.
 
 \let\OLDthebibliography\thebibliography
 \renewcommand\thebibliography[1]{
  \OLDthebibliography{#1}
  \setlength{\parskip}{2pt}
  \setlength{\itemsep}{2pt plus 0.3ex}
}

 \end{document}